\documentclass[12pt]{article}%
   
\usepackage{amsmath,enumerate}
\usepackage{amsfonts}
\usepackage{amssymb}
\usepackage{color}

\newtheorem{theorem}{Theorem}[section]

\newtheorem{corollary}[theorem]{Corollary}

\newtheorem{lemma}[theorem]{Lemma}

\newtheorem{problem}[theorem]{Problem}

\newenvironment{proof}[1][Proof]{\noindent\textbf{#1.} }
{\hfill \ \rule{0.5em}{0.5em}}

\title{Regular Tur\'an numbers of complete bipartite graphs}
\author{Michael Tait\thanks{Villanova University Department of Mathematics \& Statistics. Research is partially supported by the National Science Foundation grant DMS-2011553. email: \texttt{michael.tait@villanova.edu}} \and Craig Timmons\thanks{Department of Mathematics and Statistics, 
	California State University Sacramento, \texttt{craig.timmons@csus.edu}.
Research supported in part by Simons Foundation Grant \#359419.}}

\begin{document}
\maketitle

\begin{abstract}
    Let $\textup{rex}(n, F)$ denote the maximum number of edges in an $n$-vertex graph that is regular and does not contain $F$ as a subgraph. We give lower bounds on
    $\textup{rex}(n, F)$, 
    that are best possible up to a constant factor, when $F$ is one of $C_4$, $K_{2,t}$, $K_{3,3}$ or $K_{s,t}$ for $t>s!$. 
\end{abstract}

\section{Introduction}
For a fixed graph $F$, the {\em Tur\'an number} of $F$ is the maximum number of edges in an $n$-vertex graph that does not contain $F$ as a subgraph and is denoted by $\textup{ex}(n, F)$. Tur\'an numbers for various graphs or families of graphs are the central functions in extremal graph theory. In this paper, we study a related function, where one restricts to regular graphs.

Let $\textup{rex}(n, F)$ be the maximum number of edges
in an $n$-vertex regular $F$-free graph. 
Following \cite{CT} and \cite{worm}, we call this the {\em regular Tur\'an number} of $F$. By the definitions, we have the trivial inequality
\[
\textup{rex}(n, F) \leq \textup{ex}(n, F),
\]
for all $F$ and $n$. However, unlike $\textup{ex}(n, F)$, the function $\textup{rex}(n, F)$ is not necessarily monotone in $n$. For example, Mantel's theorem shows that $\textup{rex}(2k, K_3) = k^2$, but Andr\'asfai \cite{A} proved that $\textup{rex}(2k+1, K_3) \leq (2k+1)^2/5$.

Most of the previous work on regular Tur\'an numbers is given by the extensive study of {\em cages} (see \cite{EJ} for a survey), where one forbids all cycles up to a fixed length. For other graphs, regular Tur\'an numbers were introduced and studied systematically in \cite{worm}. The regular Tur\'{a}n problem
was motivated by Caro and Tuza's work on singular Tur\'{a}n numbers \cite{singular}.
In \cite{worm}, 
Gerbner, Patk\'{o}s, Tuza, and Vizer 
showed that for non-bipartite $F$ with odd girth $g$, one has $\textup{rex}(n, F) \geq n^2/(g+6) - O(n)$ and asked to determine $\liminf_{n\to \infty} \textup{rex}(n, F)/n^2$ for non-bipartite $F$. This problem was solved independently in \cite{CT} and \cite{CJK} for graphs $F$ with chromatic number at least $4$, and the authors proved partial results for graphs with chromatic number $3$. 
Following \cite{CT} and \cite{CJK}, exact 
results on regular Tur\'{a}n numbers of trees and complete graphs
were obtained in \cite{exact}.

In both \cite{CT} and \cite{worm} it is acknowledged that we do not know much about $\mathrm{rex}(n, F)$ when $F$ is a bipartite graph with a cycle.  This will be the focus of the current paper.  

Given a bipartite graph $F$, the quantity we will be particularly interested in is 
\[
Q(F):=\limsup_{n\to \infty} \frac{\textup{ex}(n, F)}{\textup{rex}(n, F)}
\]
which is a measure of how close $\textup{rex}(n , F)$ is
to $\textup{ex}(n ,F)$.  
By the trivial inequality, $Q(F)$ is always at least $1$. It is natural to ask when $Q(F) = 1$ and when $Q(F) < \infty$. Our main theorems give lower bounds for $\textup{rex}(n, F)$, implying that $Q(F)$ is finite for several different bipartite graphs.  We begin with 
$C_4$.

\begin{theorem}\label{main theorem c4}
The regular Tur\'{a}n number of $C_4$ satisfies
\[
\textup{rex}(n , C_4) \geq \left( \frac{1}{ 2\sqrt{6} } -o(1) \right) n^{3/2}.
\]
\end{theorem}

F\"uredi \cite{F} showed that $\mathrm{ex}(K_{2, t+1}) \sim \frac{\sqrt{t}}{2}n^{3/2}$. Since a $C_4$-free graph is also $K_{2,t+1}$-free, Theorem \ref{main theorem c4} shows that for any fixed $t$, $Q(K_{2,t+1})$ is bounded above by a constant that depends on $t$,
namely $\sqrt{t/6}$. However, just using $C_4$-free graphs does not show that $\limsup_{t\to\infty} Q(K_{2,t+1})$ is finite. We show that this limit is finite in the following theorem.

\begin{theorem}\label{main theorem k2t}
For $t \geq 1$ with $t$ even, 
the regular Tur\'{a}n number of $K_{2,2t+1}$ satisfies
\[
\textup{rex}(n , 
K_{2,2t+1}) \geq \left( \sqrt{t / 20 } -o(1) \right) n^{3/2}.
\]
\end{theorem}



Using the $K_{3,3}$-graphs constructed 
by Brown \cite{brown}, together 
with some number theoretic results 
on the Waring-Goldbach problem, we
can prove a lower bound on 
the regular Tur\'{a}n number of $K_{3,3}$.  

\begin{theorem}\label{main theorem k33}
For large enough $n$, the regular Tur\'{a}n number of $K_{3,3}$ satisfies
\[
\textup{rex}(n , K_{3,3}) 
\geq 
 \frac{1}{2 \sqrt[3]{14^2} } n^{5/3} - O(n^{3/5}).
 \]
\end{theorem}

One comment is that if $n$ is of the 
form $n = p^3$ where $p$ is an odd 
prime, then 
$\textup{rex}(n , K_{3,3} ) \geq \frac{n^{5/3} - n^{4/3} }{2}$ which is asymptotically best
possible since $\textup{ex}(n , K_{3,3} ) \sim \frac{1}{2} n^{5/3}$.  Therefore, 
\[
\liminf_{n \rightarrow \infty} 
\frac{ \textup{ex}(n , K_{3,3} ) }{ \textup{rex}(n , K_{3,3} )} =1.
\]

Finally, we use the norm graphs of Koll\'ar, R\'onyai, and Szab\'o \cite{KRS} to give lower bounds on the regular Tur\'an number of $K_{s,t}$ when $t > s!$.

\begin{theorem}\label{main theorem kst}
Let $s\geq 3$ and $t>s!$. Then there is a constant $C$ depending only on $s$ so that for large enough $n$, the regular Tur\'{a}n number of $K_{s,t}$ satisfies
\[
\textup{rex}(n, K_{s,t}) \geq Cn^{2-1/s}.
\]
\end{theorem}

We end the introduction with two open problems that we feel are most natural to try next.

\begin{problem}
Show that $Q(C_6)$ and $Q(C_{10})$ are finite.
\end{problem}

\begin{problem}
Determine whether or not $Q(C_4) =1$.
\end{problem}

In Section \ref{preliminaries}, we give an outline of how the main theorems are proved and establish some necessary lemmas. In Sections \ref{c4 section}, \ref{k2t section}, \ref{k33 section}, and \ref{kst section} we prove our main theorems.

\section{Preliminaries}\label{preliminaries}
For $F\in \{C_4, K_{2,t}, K_{3,3}, K_{s,t}\}$ with $t>s!$, there are classical constructions of $F$-free graphs with many edges coming from geometry and algebra. To give an upper bound on the quantity
\[
Q(F) = \limsup_{n\to \infty} \frac{\textup{ex}(n, F)}{\textup{rex}(n, F)},
\]
we will give constructions of regular graphs on $n$ vertices that are $F$-free and have many edges. The first difficulty is that to bound the limit superior, we need a lower bound on $\textup{rex}(n, F)$ for arbitrary $n$. Since the function $\textup{rex}(n, F)$ is not monotone in $n$, it does not suffice to construct a sequence of $F$-free graphs which have number of vertices some function of a prime and then use a density of primes argument, as is common when proving lower bounds for $\mathrm{ex}(n, F)$. 

Therefore, given an $n$, our strategy will be to take disjoint unions of $F$-free graphs so that the number of vertices sums to $n$. As the classical constructions of $F$-free graphs are defined algebraically or geometrically, the number of vertices in these graphs is some function of a prime power (see \cite{ADG survey} for a survey on algebraically defined graphs). Because of this, we will need the following theorem on the Waring-Goldbach problem when the prime powers are restricted to being almost equal.

\begin{theorem}[Wei and Wooley \cite{WW}]\label{prime basis theorem}
Let $k\geq 3$ be a natural number and let $\theta_k = \frac{4}{5}$ if $k =3$ and $\frac{5}{6}$ if $k\geq 4$. For a prime $p$, let $\tau = \tau(k,p)$ be the integer such that $p^\tau | k$ but $p^{\tau+1} \not| k$. Define $\gamma = \gamma(k,p)$ by $\gamma(k,p) = \tau+2$ when $p=2$ and $\tau > 0$, and otherwise $\gamma(k,p) = \tau+1$. Let $R = R(k) = \prod p^\gamma$where the product is taken over all primes with $(p-1)|k$. Then for any $\varepsilon > 0$ and $\ell > 2k(k-1)$, if $n$ is a sufficiently large integer congruent to $\ell$ mod $R$, the equation 
\[
n = p_1^k + p_2^k + \cdots + p_\ell^k
\]
has a solution in prime numbers $p_j$ with $|p_j - (n/\ell)^{1/k}| < \left((n/\ell)^{1/k})\right)^{\theta_k + \varepsilon}$.
\end{theorem}

The next difficulty we encounter is that when we take a disjoint union of $F$-free graphs, each component may not be regular of the same degree (indeed, each component may itself not be regular, but we ignore this for the moment). To make the vertices have the same degree, we will remove a regular subgraph from each component.  In the case of $K_{3,3}$, 
each graph in the disjoint union will be regular with all vertices having even degree, and so by Petersen's $2$-factor theorem the edges of each component can be partitioned into $2$-factors. Therefore, we may remove a spanning regular subgraph of the appropriate degree from each component in order to make the whole graph regular.

The argument for $K_{s,t}$ with 
$t > s!$ is more involved because, as mentioned earlier, each component will not be regular. We use Theorem \ref{prime basis theorem} to take a disjoint union of norm graphs, which are almost regular but have absolute points of degree $1$ fewer than the rest. We would like to remove a spanning regular subgraph of the appropriate degree from each component so that all of the absolute points in the graph have the same degree and all of the other points have degree $1$ more. Then we add a matching to the absolute points in a way that preserves $K_{s,t}$-freeness, making the whole graph regular. To do this we will iteratively remove Hamilton cycles from the graph. Removing $2$-factors iteratively would be just as good for our purpose, but we explain below why we are doing something seemingly much harder.

Much of the previous work on finding $k$-factors in graphs takes place in the setting where the host graph is regular. The main result of \cite{sasha} allows the host graph to contain loops, and so applies in our setting. If we use this theorem, we do not know how many loops were contained in the $k$-factor, and so unfortunately we cannot complete the last step in the proof (adding a matching to the absolute points). We also note that a result of Alon, Freidland, and Kalai \cite{AFK} finds regular subgraphs in almost regular graphs. However, it is crucial for our purpose that these subgraphs be spanning, which is not guaranteed by their theorem. 

To overcome these difficulties, we will use use a spectral approach.  The following result guarantees Hamilton cycles in graphs with a reasonable spectral gap. The combinatorial Laplacian of a graph is the matrix $D-A$ where $D$ is the diagonal degree matrix and $A$ is the adjacency matrix.

\begin{theorem}[Butler and Chung \cite{BC}]\label{hamiltonian theorem}
Let $G$ be a graph on $n$ vertices with average degree $d$ and $ 0 = \mu_1 \leq \mu_1 \leq \cdots \leq \mu_{n-1}$ be the eigenvalues of the combinatorial Laplacian of $G$. There is a constant $c$ so that if 
\[
|d - \mu_i| \leq c\frac{(\log \log n)^2}{\log n(\log \log \log n)}d,
\]
for all $i\not=0$ and $n$ sufficiently large, then $G$ is Hamiltonian.
\end{theorem}

We note that much smaller spectral gaps guarantee connectivity of a graph, for example Theorem 4.3 of \cite{KSsurvey}. It is possible that one could use this connectivity to verify that the proofs of theorems on finding $2$-factors in regular graphs will hold in our setting with a graph with loops. However, to keep the proof of Theorem \ref{main theorem kst} transparent, we will use Theorem \ref{hamiltonian theorem} of Butler and Chung even though it is perhaps stronger than what we need.

As a final note on the proof, while our result on the spectral gap of norm graphs (Theorem \ref{norm graph spectral gap theorem}) is much stronger than is necessary, we believe that it is interesting in its own right, adding to numerous papers which calculate the eigenvalues of algebraically defined graphs in extremal graph theory (see e.g. \cite{alon2, AR, BIP, CLL, MSW, szabo}).

To compute eigenvalues, we will need the following result on cyclotomic periods. This lemma is known but we provide a proof for completeness.

\begin{lemma}\label{character sum lemma}
Let $q$ be a prime power, $\chi$ an additive character of $\mathbb{F}_q$, and $H$ a multiplicative subgroup
of $\mathbb{F}_q^*$. Then 
\[
\left| \sum_{x\in H} \chi(x) \right| \leq \sqrt{q}.
\]
\end{lemma}
\begin{proof}
Let $\gamma$ generate $\mathbb{F}_q^*$, and let $H$ be generated by $\gamma^h$ where $h|q-1$. Consider the multiplicative character $\theta$ defined by 
\[
\theta(\gamma^k) = \textup{exp}\left(\frac{2\pi i}{h}\cdot k\right).
\]
Then the linear combination $1 + \theta + \theta^2 + \cdots + \theta^{h-1}$ evaluates to $h$ on $H$, and to $0$ on the complement of $H$. Therefore,  
\[
\sum_{x\in H} \chi(x) = \frac{1}{h}\sum_{x\in \mathbb{F}_q^*} (1+\theta + \cdots + \theta^{h-1})(x) \cdot \chi(x) = \frac{1}{h}\sum_{j=0}^{h-1} \sum_{x\in \mathbb{F}_q^*} \theta^j(x)\cdot \chi(x).
\]
Standard theorems on Gauss sums (e.g. Theorem 5.11 in \cite{lidl}) give that each inner sum has modulus bounded by $\sqrt{q}$. Using the triangle inequality completes the proof.
\end{proof}


\section{The regular Tur\'{a}n number of $C_4$}\label{c4 section}
In this section we prove Theorem \ref{main theorem c4}. Given an $n$, if $n$ is even our strategy will be to construct a bipartite Cayley graph that is $C_4$-free. 
This is the easy case.  
If $n$ is odd, we will use a construction from geometry to find a regular $C_4$-free graph on an odd number of vertices.  Then we take this graph and a disjoint union of a bipartite Cayley graph to find a $C_4$-free graph on $n$ vertices. We will choose the size of the generating set of the Cayley graph to ensure that the entire graph is regular. Now for the details.

We begin  with a lemma that is most 
certainly known.  It proves 
that for any large enough integer $M$, there is an 
$M \times M$ $C_4$-free bipartite graph 
that is $k$-regular where $k$ may be 
taken asymptotically as large as $(M/2)^{1/2}$.  

\begin{lemma}\label{bipartite c4 lemma}
There is an integer $n_0$ such that the following holds.  For any integer $M > n_0$, there 
is an $M \times M$ bipartite $C_4$-free graph that is $k$-regular for any $k$ with 
\[
1 \leq k \leq \left( \lfloor M/2 \rfloor + 1 \right)^{1/2} - 
\left( \lfloor M/2 \rfloor +1 \right)^{0.2625}.
\]
\end{lemma}
\begin{proof}
By a result of Baker, Harman, and Pintz \cite{baker}, there is an $x_0$ such that for all $x > x_0$, 
the interval $[x - x^{0.525} , x]$ contains a prime.  We apply this result to obtain that 
for $M > n_0$, there is a prime $p$ with 
\begin{equation}\label{bipartite c4 lemma eq1}
\left( \lfloor M/2 \rfloor + 1 \right)^{1/2} - 
\left( \lfloor M/2 \rfloor +1 \right)^{0.2625} \leq p \leq \left( \lfloor M/2 \rfloor + 1 \right)^{1/2}.
\end{equation}
Let $\mathcal{A} \subset \mathbb{Z}_{p^2 - 1}$ be a Bose-Chowla Sidon set (see \cite{bc}).
Thus, $|\mathcal{A}| = p$, and for any 
$a_1,a_2,a_3,a_4 \in \mathcal{A}$, the equation
$a_1 + a_2 \equiv a_3 +a_4 ( \textup{mod}~p^2 - 1 )$
implies $\{a_1 , a_2 \} = \{ a_3 , a_4 \}$.  
Inequality (\ref{bipartite c4 lemma eq1}) implies that $p^2 - 1 \leq \lfloor M / 2 \rfloor$.  
Since $\mathcal{A} \subset \mathbb{Z}_{p^2 - 1} = \{1,2, \dots , p^2 - 1 \}$, we 
may view $\mathcal{A}$ as a subset of $\{1,2, \dots , \lfloor M / 2 \rfloor \}$.  
For any $k$ with $k \leq  \left( \lfloor M/2 \rfloor + 1 \right)^{1/2} - 
\left( \lfloor M/2 \rfloor +1 \right)^{0.2625}$, we may choose a subset $A \subseteq \mathcal{A}$ 
with $|A | = k$.  Define an $M \times M$ bipartite graph with parts $X = \mathbb{Z}_M$
and $Y = \mathbb{Z}_M$ where $x \in X$ is adjacent to $y \in Y$ if and only if 
$x + y \equiv a ( \textup{mod}~M)$ for some $a \in A$.  This bipartite graph 
has parts of size $M$, and is regular of degree $|A| = k$.  We finish the proof 
of the lemma by showing that this graph is $C_4$-free.

Suppose $x_1y_1x_2 y_2$ is a 4-cycle with $x_1,x_2 \in X$, $y_1 , y_2 \in Y$.  There are 
elements $a,b,c,d \in A$ such that 
\begin{center}
$x_1 + y_1 \equiv a ( \textup{mod}~M)$, ~~~~~~~~~~ $x_1 + y_2 \equiv b ( \textup{mod}~M)$, 

\smallskip

$x_2 + y_1 \equiv d ( \textup{mod}~M)$, ~~~~~~~~~~$x_2 + y_2 \equiv c ( \textup{mod}~M)$.
\end{center}
This system of congruences implies 
\[
a - b + c - d \equiv 0 ( \textup{mod}~M) ~~~ \Rightarrow ~~~
a + c \equiv b + d ( \textup{mod}~M).
\]
Recalling that $\mathcal{A}$, hence $A$, is contained in $\{1,2, \dots , \lfloor M / 2 \rfloor \}$, 
this last congruence can be turned into an equality in $\mathbb{Z}$ so
$a + c = b +d$.  Taking this equation modulo $p^2 - 1$ and using the fact 
that $A$ is a Sidon set gives $\{a ,c \} = \{ b , d \}$.  If $a \equiv b ( \textup{mod}~p^2 - 1 )$, then 
$a \equiv b ( \textup{mod}~M)$ which implies $y_1$ and $y_2$ are the same vertex.  A similar contradiction 
occurs if $a \equiv d ( \textup{mod}~M)$.  
\end{proof}

\begin{corollary}\label{corollary c4 even}
There is an $n_0$ such that for all $n>n_0$
with $n$ even, there is a $C_4$-free $n$-vertex graph that is $( \lfloor n/2 \rfloor +1)^{1/2} - ( \lfloor n/2 \rfloor +1)^{0.2625}$-regular.
\end{corollary}

For odd $n$, our lower bound will be obtained by taking the disjoint union 
of two $C_4$-free graphs.  One of these graphs is an induced subgraph
of the Erd\H{o}s-R\'{e}nyi orthogonal polarity graph
$ER_q$.  

Let $q$ be a power of an 
odd prime.  Parsons \cite{parsons} (see also \cite{williford}) showed that 
there is a $C_4$-free $\frac{q-1}{2}$-regular graph on $\binom{q+1}{2}$ vertices, and 
another on $\binom{q}{2}$ vertices.  We denote these graphs by $R_{1,q}$ and 
$R_{2,q}$, respectively. These are 
induced subgraphs of $ER_q$ and more 
on these subgraphs, and $ER_q$ in general, 
can be found in Williford's Ph.D.\ thesis \cite{williford}.  

\begin{theorem}\label{theorem c4 odd}
Let $0 < \epsilon < \frac{1}{100}$.  There is an 
$n_0  = n_0 ( \epsilon )$ such that the following holds.  
For all odd $n > n_0$,
there is a $C_4$-free $n$-vertex graph that is 
$\left( \sqrt{1/6} - \epsilon \right) n^{1/2}$-regular.
\end{theorem}
\begin{proof}
Let $0 < \epsilon < \frac{1}{100}$ and let $\delta = \frac{2}{3} - \epsilon$.  Let $n$ be large enough so that there 
is a prime $p$ with
\begin{equation}\label{theorem c4 odd eq1}
\lfloor \sqrt{ \delta n} \rfloor - \lfloor \sqrt{ \delta n } \rfloor^{0.525} \leq p \leq 
\lfloor \sqrt{ \delta n } \rfloor .
\end{equation}
Define 
\[
N =
\left\{  
\begin{array}{ll}
n - \binom{p+1}{2} & \mbox{if $p \equiv 1 ( \textup{mod}~4)$}, \\
n - \binom{p}{2} & \mbox{if $p \equiv 3 ( \textup{mod}~4)$}.
\end{array}
\right.
\]
Since $n$ is odd, $N$ is even by definition and we let $N = 2M$.  We will now assume that $p \equiv 1 ( \textup{mod}~4)$
as the proof in the case when $p \equiv 3 ( \textup{mod}~4)$ is similar.  

The graph we construct will be the disjoint union of two graphs, one of which is 
a bipartite graph obtained from applying Lemma \ref{bipartite c4 lemma}.  The other is $R_{1,p}$, which 
has $\binom{p+1}{2}$ vertices and is $\frac{p-1}{2}$-regular.  
We wish to apply Lemma \ref{bipartite c4 lemma} to obtain a $N/2 \times N/2$ bipartite graph $B_1$ that has 
$N = n - \binom{p+1}{2}$ vertices and is $\frac{p-1}{2}$-regular.  To do so, we need 
\begin{equation}\label{theorem c4 odd eq2}
1 \leq \frac{p-1}{2} \leq \left( \lfloor N/4 \rfloor + 1 \right)^{1/2} 
- 
\left(  \lfloor N/4 \rfloor \right)^{ 0.2625}.
\end{equation}
By (\ref{theorem c4 odd eq1}), $p = (1 + o(1))   \sqrt{ ( 2/3 - \epsilon ) n }$.
By definition of $N$, 
\[
N= n - (1 + o(1)) \frac{p^2}{2} = n - (1 + o(1)) \frac{ (2/3 - \epsilon ) }{2} n
= 
\left( \frac{2}{3} + \frac{ \epsilon }{2}  +o(1) \right)  n.
\]
Thus, the right hand side of (\ref{theorem c4 odd eq2}) is 
\[
\left( \lfloor N/4 \rfloor + 1 \right)^{1/2} 
- 
\left(  \lfloor N/4 \rfloor \right)^{ 0.2625}=
(1 + o(1)) \frac{1}{2} \sqrt{( 2/3 + \epsilon /2 +o(1) ) n},
\]
which, with 
\[
\frac{p-1}{2} = (1 + o(1)) \frac{1}{2} \sqrt{ ( 2/3 - \epsilon) n},
\]
shows that (\ref{theorem c4 odd eq2}) holds for large enough $n$.  By Lemma \ref{bipartite c4 lemma}, 
there is a $\frac{p-1}{2}$-regular $N/2 \times N/2$ bipartite graph that is $C_4$-free. 
Taking the disjoint union of this bipartite graph and $R_{1,p}$ proves Theorem \ref{theorem c4 odd} in 
the case $p \equiv 1 ( \textup{mod}~4)$.  
\end{proof}


\section{The regular Tur\'{a}n number of $K_{2,t}$}\label{k2t section}

In this section we prove Theorem \ref{main theorem k2t}. Similar to the previous section, our strategy will be to use either a bipartite Cayley graph, or the disjoint union of a regular graph on an odd number of vertices and a bipartite Cayley graph. The first step is to construct a graph similar to the $K_{2,t+1}$-free graphs of F\"uredi \cite{F}.  While F\"uredi's constructions are algebraically defined graphs, ours will be written as Cayley sum graphs. We note that in \cite{mors} and \cite{zoli} bipartite $K_{2,t}$-free graphs of the same flavor are constructed.

Let $p$ be an odd prime and let $\theta$ be a generator of the multiplicative group $\mathbb{F}_p^*$.  Suppose 
$t \geq 1$ is an integer that divides $p-1$.  Let $\Gamma = \mathbb{Z}_{ (p-1)/t } \times \mathbb{F}_p$ 
and $\mu = \theta^{ \frac{ p -1}{t} }$.  
Define 
\[
S  = \{ ( m , \theta^m \mu^{n} ) : m \in \mathbb{Z}_{(p-1)/t} , 0 \leq n \leq t-1 \}.
\]
The set $S$ can also be written as 
\[
S = \left\{ \left( a \left( \textup{mod}~ \frac{p-1}{t} \right) , \theta^a ( \textup{mod}~p) \right) : a \in \mathbb{Z}_{p-1} \right\},
\]
where we use the least residues $\mathbb{Z}_{p-1} = \{0,1, \dots , p - 2 \}$.  

Let $H_{p , t, }$ be the graph with vertex set $\Gamma$, and distinct vertices
$(x,y)$ and $(a,b)$ are adjacent if 
\[
(x ,y) + (a,b) \in S.
\]
Thus, $(x,y)$ and $(a,b)$ are adjacent if and 
only if there is an $m \in \mathbb{Z}_{ (p-1)/t } $ and $n \in \{0,1, \dots , t - 1 \}$ 
such that 
\begin{center}
$x + a \equiv m ( \textup{mod}~ \frac{p-1}{t} )$ ~~~~~
and 
~~~
$y + b \equiv \theta^m \mu^n ( \textup{mod}~p)$.
\end{center}
This graph is a modification of a graph 
constructed by Ruzsa \cite{ruzsa}.  

\begin{lemma}\label{ruzsa k2t free}
The graph $H_{p,t}$ is $K_{2,t+1}$-free.
\end{lemma}
\begin{proof}
Suppose $(x,y)$ and $(u,v)$ are two distinct vertices with 
$t+1$ common neighbors $(s_i , w_i )$, $1 \leq i \leq t + 1$.  
There are elements $a_i , b_i \in \mathbb{Z}_{p-1}$ such that 
\begin{center}
$x + s_i \equiv a_i ( \textup{mod}~ \frac{p-1}{t} )$, ~~~~~~~~~~
$y + w_i \equiv \theta^{a_i } ( \textup{mod}~p)$, \\
\smallskip
$u + s_i \equiv b_i ( \textup{mod}~ \frac{p-1}{t} )$, ~~~~~~~~~~
$v  + w_i \equiv \theta^{b_i} ( \textup{mod}~p)$.
\end{center}
Therefore, $x - u \equiv a_i - b_i ( \textup{mod}~ \frac{p-1}{t} )$ 
and $y - v \equiv \theta^{a_i} - \theta^{b_i} ( \textup{mod}~p)$.  
The first congruence implies that there is an integer $\delta_i$ such that 
$x - u = a_i - b_i + \delta_i \left( \frac{p-1}{t} \right)$ in $\mathbb{Z}$.  Hence,
\[
\theta^{ x - u } \equiv \theta^{ a_i - b_i + \delta_i ( (p-1)/t) } 
\equiv 
\theta^{a_i - b_i } \mu^{ \delta_i } ( \textup{mod}~p).
\]
The exponent $\delta_i$ may be taken modulo $t$ since $\mu = \theta^{ (p-1)/t}$, and 
so we let $\delta_i^* \equiv \delta_i ( \textup{mod}~t )$ where $\delta_i^* \in \{0,1, \dots , t- 1 \}$.
Since $i$ ranges from 1 to $t+1$, there exists $i,j$ with 
$1 \leq i < j \leq t + 1$ 
and $\delta_i^* = \delta_j^*$.  This gives
\[
\theta^{ a_i - b_i} \mu^{ \delta_i^*} \equiv \theta^{ a_j - b_j} \mu^{ \delta_j^*} ( \textup{mod}~p)
\]
so $\theta^{a_i - b_i} \equiv \theta^{a_j - b_j} ( \textup{mod}~p)$.  
Let $A = \theta^{a_i} \theta^{b_j} = \theta^{a_j} \theta^{b_i}$.  
Using the fact that 
$y - v \equiv \theta^{a_i} - \theta^{b_i} ( \textup{mod}~p)$, we let 
\begin{center}
$B = \theta^{a_i} + \theta^{b_j} \equiv \theta^{a_j} + \theta^{ b_i } ( \textup{mod}~p)$.
\end{center}
The pairs $\{ \theta^{a_i} , \theta^{b_j} \}$ and $\{ \theta^{a_j} , \theta^{b_i} \}$ 
are the roots of $X^2 - BX + A$ in $\mathbb{F}_p$.  By unique factorization in $\mathbb{F}_p[x]$, 
$\{ \theta^{a_i} , \theta^{b_j} \} = \{ \theta^{a_j} , \theta^{b_i} \}$.  
If $a_i \equiv a_j ( \textup{mod}~p)$ and $b_j \equiv b_i ( \textup{mod}~p)$, 
then the vertices $(s_i , w_i)$ and $(s_j , w_j)$ are the same, a contradiction.
If $a_i \equiv b_i ( \textup{mod}~p)$ 
and $b_j \equiv a_j ( \textup{mod}~p)$, then the vertices $(x,y)$ and $(u,v)$ are the same, 
another contradiction.  This shows $H_{p,t}$ is $K_{2,t+1}$-free.  
\end{proof}

\begin{lemma}\label{ruzsa absolute points}
The graph $H_{p,t}$ contains $p-1$ vertices of degree $p-2$, and all other vertices 
have degree $p-1$.
\end{lemma}
\begin{proof}
Let $(x,y)$ be a vertex in $H_{p,t}$.  Then, since $|S| = p-1$, the vertex $(x,y)$ has degree $p-1$ unless
\[
(x,y) + (x,y) = ( m , \theta^m \mu^n )
\]
for some $m \in \mathbb{Z}_{ (p-1)/t }$ and $n \in \{0,1, \dots , t-1 \}$.
From $x + x \equiv m ( \textup{mod} ~ \frac{p-1}{t} )$, we get $2x = m + \delta \left( \frac{p-1}{t} \right)$
for some integer $\delta$.  Then 
\[
2y \equiv \theta^{2x - \delta ( (p-1)/t ) } \mu^{n} \equiv \theta^{2x} \mu^{ n - \delta } ( \textup{mod}~p).
\]
Therefore, 
\[
(x,y) = (x , 2^{-1} \theta^{2x} \mu^{n - \delta} ).
\]
There  are $\frac{p-1}{t}$ choices for $x$ and $t$ choices for $n - \delta$ (this exponent can be taken 
modulo $t$ since $\mu = \theta^{ (p-1)/t}$) which gives $p-1$ vertices of degree $p-2$.  Conversely, one can check 
that any vertex of the form $(x, 2^{-1} \theta^{2x} \mu^r )$ with $r \in \{0,1, \dots , t - 1 \}$ 
will have degree $p-2$. 
\end{proof}

\bigskip

Following the standard terminology, vertices of degree $p-2$ in $H_{p,t}$ are called \emph{absolute points}.
Let $H_{p,t}^*$ be the supergraph of $H_{p,t}$ obtained by adding a new vertex $a$ that 
is adjacent to all of the absolute points 
of $H_{p,t}$.
By Lemma \ref{ruzsa absolute points}, the graph $H_{p,t}^*$ has $1 + \frac{ p (p-1)}{t}$ 
vertices and is $(p-1)$-regular.  We now show that $H_{p,t}^*$ is $K_{2 , 2t +1}$-free.

\begin{lemma}\label{star}
The graph $H_{p,t}^*$ is $K_{2,2t+1}$-free.
\end{lemma}
\begin{proof}
By Lemma \ref{ruzsa k2t free}, any $K_{2,2t+1}$ in $H_{p,t}^*$ must use the added vertex $a$.  
We will show that in $H_{p,t}$, no vertex is adjacent to $2t +1$ absolute points and
so $H_{p,t}^*$ will be $K_{2,2t+1}$-free.

Suppose $(x,y)$ is a vertex in $H_{p,t}$ that is adjacent to 
absolute points $(z_i , 2^{-1} \theta^{2z_i} \mu^{r_i})$ 
for some $1 \leq i \leq D$,
where $z_i \in \mathbb{Z}_{ \frac{p-1}{t}}$ and 
$r_i \in \mathbb{Z}_t$.  
We must show $D \leq 2t$.    
Then we have
\[
(x , y) + ( z_i , 2^{-1} \theta^{ 2 z_i} \mu^{ r_i } ) = ( m_i , \theta^{m_i} \mu^{n_i} )
\]
for some $m_i \in \mathbb{Z}_{ \frac{p-1}{t}}$ and $n_i \in \mathbb{Z}_t$.    
Thus, $x+ z_i \equiv m_i ( \textup{mod} ~ \frac{p-1}{t} )$ so $x + z_i = m_i + \delta_i \left( \frac{p-1}{t} \right)$ 
for some integer $\delta_i$.  This last equation, with
\[
y + 2^{-1} \theta^{2 z_i} \mu^{r_i} - \theta^{m_i} \mu^{n_i} \equiv 0 ( \textup{mod}~p),
\]
implies 
\[
y + 2^{-1} \theta^{2 z_i} \mu^{r_i} - \theta^{x+z_i} \mu^{n_i - \delta_i} \equiv 0 ( \textup{mod}~p).
\]
Therefore, $\theta^{z_i}$ is a root of the degree 2 polynomial  
\[
f (X) = 2^{-1} \mu^{r_i} X^2 - \theta^x \mu^{n_i - \delta_i} X + y.
\]
 There are $t$ choices for $r_i$ and then at most 2 choices for $\theta^{z_i}$ since $f(X)$ has at most two roots.
Hence, $(x,y)$ is adjacent to at most $2t$ absolute points,
so $D \leq 2t$.  
\end{proof}

\begin{corollary}\label{star corollary}
Let $p$ be a prime and $t \geq 1$ be an integer for which $t$ divides $p-1$.  Then the graph 
$H_{p,t}^*$ is a $(p-1)$-regular graph with $\frac{p (p-1)}{t} + 1$ vertices and is
$K_{2,2t+1}$-free.
\end{corollary}

The next step is to prove a version of Lemma \ref{bipartite c4 lemma} for $K_{2,2t+1}$.  For this we use a set 
constructed in \cite{tt} that was used to solve a bipartite
Tur\'{a}n problem in $k$-partite graphs.  

Let $q$ be a power of an odd prime and suppose $t \geq 1$ is an integer for which $t$ divides $q -1$.  
Let $H$ be the subgroup of $\mathbb{Z}_{ (q^2 - 1)/t}$ generated by $\frac{ q-1}{t} ( q + 1) = \frac{q^2 - 1 }{t}$.
Assume $\theta$ is a generator of $\mathbb{F}_{q^2}^*$ and let 
$\mathcal{A} = \{ a \in \mathbb{Z}_{ q^2 - 1} : \theta^a - \theta \in \mathbb{F}_q \}$ be a Bose-Chowla Sidon set \cite{bc}.  
Let $\Gamma$ be the quotient group 
$\mathbb{Z}_{q^2 - 1} / H \cong \mathbb{Z}_{ (q^2 - 1)/t}$.  Finally, in the quotient 
group $\Gamma$, let $A$ be the set defined by 
\[
A   = \{ a +H : a \in \mathcal{A} \}.
\]
In \cite{tt} it is shown that $|A| = q$, and that 
for any nonzero $\alpha \in \Gamma$, the number of ordered pairs $(a,b) \in A \times A$ 
for which $\alpha = a - b $ in $\Gamma$ is at most $t$.  We state this as a lemma.  

\begin{lemma}[\cite{tt}]\label{k2t lemma}
If $t \geq 1$ and $q$ is a power of an odd prime with $q \equiv 1 ( \textup{mod}~t)$, 
then there is a set $A \subset \mathbb{Z}_{(q^2 - 1)/t}$ with $|A| = q$ such that for 
any $\alpha \in \mathbb{Z}_{(q^2 - 1)/t } \backslash \{ 0 \}$, 
there are at most $t$ ordered pairs $(a,b) \in A \times A$ such that \[
a-b \equiv \alpha \left( \textup{mod}~ \frac{q^2-1}{t} \right).
\]
\end{lemma}

\begin{lemma}\label{bipartite k2t lemma}
Let $t \geq 1$ be an integer and $\epsilon > 0$ be a positive real number.  
There is an $n_0 = n_0(t , \epsilon )$ such that the 
following holds.  For any $M > n_0$ and $k$ with 
\[
1 \leq k \leq  (1 - \epsilon ) \left(  \left\lfloor \frac{tM}{2} \right\rfloor +1 \right)^{1/2},
\]
there is a $k$-regular $M \times M$ bipartite graph that is $K_{2,2t+1}$-free 
\end{lemma}
\begin{proof}
Let $t \geq 1$.  Choose a prime $p$ with $p \equiv 1 ( \textup{mod}~t)$ and 
\begin{equation}\label{another eq1}
(1 - \epsilon ) \left( \lfloor tM/2 \rfloor + 1 \right)^{1/2} 
\leq p 
<
 \left( \lfloor tM/2 \rfloor + 1 \right)^{1/2}.
\end{equation}
This can be done by Dirichlet's Theorem on primes in arithmetic progressions, and in 
particular the Siegel-Walfisz Theorem.  Indeed, the Siegel-Walfisz Theorem gives 
that for large enough $M$, the number of primes $p$ with $p \equiv 1 ( \textup{mod}~t)$ satisfying (\ref{another eq1}) is 
at least 
\[
\frac{  \epsilon ( tM/2)^{1/2} }{ 2 \phi (t) \ln ( tM / 2 ) }  - 
O \left(  \frac{ M^{1/2} }{ \ln^2 (M) } \right).
\]  
This is positive for large enough $M$ depending on $\epsilon$ and $t$.  
Here $\phi (t)$ is Euler's totient function.  
Let $A \subset \mathbb{Z}_{ (p^2 - 1)/t}$ be as in Lemma \ref{k2t lemma}.  By 
(\ref{another eq1}), we may view $A \subset \{1,2, \dots , \lfloor M / 2 \rfloor \}$ 
since $\frac{p^2 - 1}{t} < \lfloor \frac{M}{2} \rfloor$.  For any $k$ with 
\[
1 \leq k \leq (1 - \epsilon ) \left(  \lfloor tM /2 \rfloor +1 \right)^{1/2},
\]
we may choose a subset $A' \subseteq A$ with $|A'| = k$.  Thus, $A'$ is a $k$ element subset of 
$\mathbb{Z}_M$ that is contained in the ``first half" $\{1,2, \dots , \lfloor M / 2 \rfloor \}$ of $\mathbb{Z}_M$.  

Define an $M \times M$ bipartite graph with parts $X = \mathbb{Z}_M$ and $Y = \mathbb{Z}_M$ where
$x \in X$ is adjacent to $y \in Y$ if and only if 
\[
x + y \equiv a ( \textup{mod}~M)
\]
for some $a \in A'$.  This graph is $k$-regular.  We complete the proof by showing that it is $K_{2,2t+1}$-free.
Let $x_1, x_2$ be distinct vertices in $X$, say with $1 \leq x_2 < x_1 \leq M$, and suppose 
this pair of vertices is adjacent $2t+1$ vertices $y_1 , y_2 , \dots , y_{2t+1} \in Y$.  Then 
there are elements $a_i , b_i \in A'$ such that 
\begin{center}
$x_1 + y_i \equiv a_i ( \textup{mod}~M)$ ~~~~~ 
and ~~~~~
$x_2 + y_i \equiv b_i ( \textup{mod}~M)$
\end{center}
for $1 \leq i \leq 2t + 1$.    
Hence, 
\begin{equation}\label{another eq2}
x_1 - x_2 \equiv a_i - b_i ( \textup{mod}~M)\
\end{equation}
for each $i$.  Now 
$x_1 -x_2 \in \{1,2 , \dots , M-1 \}$, and since $A' \subset \{1,2, \dots , \lfloor M /2 \rfloor \}$, 
we know $- \lfloor M /2 \rfloor < a_i - b_i < \lfloor M /2 \rfloor$.  Thus, from (\ref{another eq2}) we get 
\[
x_1 - x_2 = a_i - b_i + \delta_i M
\]
where $\delta_i \in \{0,1 \}$.  If $\delta_i = 0$ for $t+1$ distinct $i$, say $1 \leq i \leq t + 1$, 
then $x_1 - x_2 = a_i - b_i$ (in $\mathbb{Z}$) which gives $x_1 -x_2 \equiv a_i - b_i ( \textup{mod}~ \frac{p^2 - 1}{t} )$.  
By our assumption on $A$, this forces $x_1 \equiv x_2 ( \textup{mod}~ \frac{p^2 - 1 }{t} )$ 
and so $a_i  = b_i$.  Combining this with (\ref{another eq2}) gives $x_1 \equiv x_2 ( \textup{mod}~M)$ 
which is a contradiction because $x_1 $ and $x_2$ are distinct vertices.  
Now assume $\delta_i = 1$ for $t+1$ distinct $i$, again say $1 \leq i \leq t + 1$.  This gives 
$x_1 - x_2 - M  = a_i - b_i$ (in $\mathbb{Z}$) and so 
$x_1 - x_2 - M \equiv a_i - b_i ( \textup{mod}~ \frac{p^2 - 1}{t} )$.  This 
congruence gives a similar contradiction as before.  
The conclusion is that this bipartite graph is indeed $K_{2,2t+1}$-free.  
\end{proof}  

\begin{corollary}\label{bipartite k2t corollary}
Let $t \geq 1$ be an integer and $\epsilon > 0$.
There is an $n_0 = n_0 ( t , \epsilon)$ such that the following holds.  For all even $n > n_0$, there is an
$n$-vertex $K_{2,2t+1}$-free graph that is $k$-regular 
where $k \geq ( 1 - \epsilon ) ( tn/4)^{1/2}$.
\end{corollary}

The last result of this section deals with the case when 
$n$ is odd.  

\begin{theorem}
Let $t \geq 1$ be an even integer and let $\epsilon > 0$.  There is an $n_0 = n_0 ( t , \epsilon )$ such that 
for all odd $n \geq n_0$, there is a $k$-regular $n$-vertex $K_{2,2t+1}$-free graph with 
\[
k \geq (1  - 2 \epsilon )^{1/2} \sqrt{tn/5}.
\]
\end{theorem}
\begin{proof}
Let $t \geq 1$ be an even integer and write $t =2^r s$ where $r \geq 1$ and $s$ is odd.  Let $\epsilon  > 0$,   
$n > n_0$ be an odd integer, and $p$ be a prime with 
\[
(1 - 2 \epsilon )^{1/2} \sqrt{ tn /5} \leq p \leq (1 - \epsilon )^{1/2} \sqrt{ tn / 5}
\]
and
\[
 p \equiv 1 + 2^r s ( \textup{mod}~2^{r+1}s ).
\]
Such a prime exists by the Siegel-Walfisz Theorem (note $\textup{gcd}(1+2^r s , 2^{r+1} s) = 1$).  
Define $N$ by 
\[
n = \frac{p (p-1)}{t} + N.
\]
The assumption on $p$ implies that there is an integer $\alpha$ such that 
$ p -1 = 2^r s + \alpha 2^{r+1}s$.  Then 
\[
\frac{ p (p-1)}{t} = \frac{ p ( 2^r s + \alpha 2^{r+1} s ) }{ 2^r s } 
= 
p ( 1 + 2 \alpha )
\]
which is odd.  Therefore, $N$ is even, say $2M   = N$.  We now wish to apply Lemma \ref{bipartite k2t lemma} 
with $M = N/2$.  To do so, we will need 
\[
p-1 \leq ( 1 - \epsilon ) \left( \lfloor tN / 4 \rfloor + 1 \right)^{1/2}.
\]
Now 
\[
N = n - \frac{ p (p-1) }{t} \geq n - \frac{p^2}{t} \geq n  - ( 1 - \epsilon ) n /5 = \left( \frac{4}{5} + \epsilon \right)n.
\]
Thus, 
\begin{eqnarray*}
(1 - \epsilon ) \left( \lfloor tN/ 4 \rfloor + 1 \right)^{1/2} 
& \geq & (1 - \epsilon ) \left( tN / 4 \right)^{1/2} 
 \geq  (1 - \epsilon ) ( t / 4 ( 4/5 + \epsilon )n )^{1/2} \\
& = & ( 1 - \epsilon ) \left( 1/5 + \epsilon / 4 \right)^{1/2} ( tn )^{1/2} \\
&  \geq & \sqrt{ 1 - \epsilon } ( tn/5)^{1/2} > p-1
\end{eqnarray*}
where the second to last inequality follows since $\epsilon < \frac{1}{5}$, and the last inequality 
follows since $p -1 < \sqrt{1 - \epsilon } ( tn / 5)^{1/2}$.  
We apply Lemma \ref{bipartite k2t lemma} 
to obtain an $N/2 \times N/2$ bipartite graph that is $K_{2,2t+1}$-free and $(p-1)$-regular.  
Taking the disjoint union of this bipartite graph together with the graph $H_{p,t}^*$ from Corollary \ref{star corollary}
gives a $(p-1)$-regular graph on $n$ vertices that is $K_{2,2t+1}$-free.  
Finally, observe 
\[
p - 1 \geq \sqrt{ 1 - 2 \epsilon } (tn/5)^{1/2}.
\]
\end{proof}


\section{The regular Tur\'{a}n number of $K_{3,3}$}\label{k33 section}

In this section we prove Theorem \ref{main theorem k33}. The outline of the proof is to take several disjoint copies of regular $K_{3,3}$-free graphs constructed by Brown \cite{brown} and to remove a regular subgraph from each component so that the entire graph is regular. 

Let $p$ be an odd prime and write $\eta$ for the quadratic character on $\mathbb{F}_p$.
Brown \cite{brown} constructed $K_{3,3}$-free 
that
gave an asymptotically tight lower bound 
on the Tur\'{a}n number of $K_{3,3}$.  This graph,
which we denote by $B(p , \alpha )$, is 
defined as follows. 
For an odd prime $p$ and 
$\alpha \in \mathbb{F}_p \backslash \{ 0 \}$ satisfying 
$\eta ( \alpha ) = - \eta (-1)$,  
let $B(p, \alpha)$ be the graph 
with vertex set $\mathbb{F}_p^3$ where $(x,y,z)$ is adjacent to $(a,b,c)$ if 
\[
(x -a)^2  + (y-b)^2 + (z-c)^2 = \alpha.
\]

The graph $B(p, \alpha)$ is $(p^2-p)$-regular. Using Theorem \ref{prime basis theorem} gives that for sufficiently large $n$, 
there are primes $p_1, \dots , p_k$ where $k =13 $ if 
$n$ is odd, and $k = 14$ if $n$ is even, such that 
\[
n = \sum_{i=1}^{k} p_i^3 
~~~ \mbox{and}~~~
| p_j - (n/k)^{1/3} | \leq n^{4/15+\epsilon}
\]
 for $1 \leq j \leq k$.

We briefly remark that using a theorem of Kumchev and Liu \cite{kumchev} would give a better error term.

\begin{theorem}\label{theorem for k33}
Let $\epsilon>0$ be arbitrary. For all sufficiently large $n$, there is an $n$-vertex
$K_{3,3}$-free graph that is $k$-regular where 
$k \geq (n/13)^{2/3}-O(n^{3/5 + \epsilon})$ when $n$ is odd, 
and $k \geq (n/14)^{2/3} - O(n^{3/5 + \epsilon})$ when $n$ is even.  
\end{theorem}
\begin{proof}
First assume that $n$ is an odd integer large enough 
so that there are primes $p_1, \dots , p_{13}$ for which 
$n = \sum_{i=1}^{13}p_i^3$ and 
\[
| p_i - ( n / 13)^{1/3} | \leq n^{4/15 + \epsilon}
\]
for $1 \leq i \leq 13$.  We may assume that 
$p_1 \geq p_2 \geq \dots \geq p_{13}$.  
Let $G$ be the disjoint union 
of the Brown graphs $B(p_i , \alpha_i)$ for 
$1 \leq i \leq 13$ and some choice of 
$\alpha_i$.  Clearly $G$ is $K_{3,3}$-free and 
has $n$ vertices.  We now remove edges from $G$
to obtain a $(p^2_{13} - p_{13})$-regular graph.  This will complete the proof 
since 
\[
p_{13} \geq (n/13)^{1/3} - n^{4/15+\epsilon}
~~~ 
\Rightarrow
~~~ p_{13}^2 - p_{13} \geq \left( \frac{n}{13} \right)^{2/3}
- O(n^{3/5 + \epsilon}).
\]
Let $i \in \{1,2, \dots , 12 \}$.  
Consider the graph $B(p_i ,\alpha_i)$.  
For any $i$, 
\[
| p_i - p_{13} | \leq | p_1 - p_{13} | 
\leq | p_1 - (n/13)^{1/3} | + | (n/13)^{1/3} - p_{13} |
\leq 2n^{4/15 + \epsilon}.
\]
Hence, 
\begin{equation}\label{an inequality}
| (p_i^2 - p_i) - (p_{13}^2 - p_{13} ) | \leq 8 n^{9/15 + \epsilon}.
\end{equation}
Define $k_i$ by $2k_i = (p_i^2 - p_i) - (p_{13}^2 - p_{13})$
(observe $k_i$ is an integer since the right hand 
side is even).  Now each $B(p_i   , \alpha_i)$ is $(p_i^2 - p_i)$-regular so by Petersen's $2$-factor theorem
we may repeatedly remove 2-factors a total of $k_i$ times, and the 
result is that each component is $(p_{13}^2 - p_{13}$)-regular.   

The argument in the case when $n$ is even the same
with the exception that we must write 
$n$ as $\sum_{i=1}^{14} p_i^{3}$ instead of a sum
with 13 terms.  
\end{proof}


\section{The regular Tur\'an number of $K_{s,t}$ when $t>s!$}\label{kst section}
In this section, let $s$ and $t$ be fixed and $t > s!$. We will use the norm-graphs from \cite{KRS}, defined as follows. For $q$ a prime power and $a\in \mathbb{F}_{q^s}$, let $N(a)$ be the $\mathbb{F}_q$ norm of $a$, that is 
\[
N(a) = a\cdot a^q \cdot a^{q^2} \cdot \cdots \cdot a^{q^{s-1}} = a^{(q^s-1)/(q-1)} \in \mathbb{F}_q.
\]
The norm-graph has vertex set $\mathbb{F}_{q^s}$ and $a\sim b$ if $N(a+b) =1$. If $N(a+a) = 1$ we call $a$ an {\em absolute point}. Let $N_{q,s}$ be the norm-graph with the loops removed from the absolute points and $N_{q,s}^o$ be the norm-graph including the loops. The number of solutions in $\mathbb{F}_{q^s}$ to the equation $N(x)=1$ is $\frac{q^s-1}{q-1}$ (see \cite{lidl} or \cite{KRS}). Therefore, the graph $N_{q,s}^o$ is $\frac{q^s-1}{q-1}$-regular (counting loops as one neighbor), and the graph $N_{q,s}$ has $\frac{q^s-1}{q-1}$ vertices of degree $\frac{q^s-1}{q-1} - 1$ and $q^s - \frac{q^s-1}{q-1}$ vertices of degree $\frac{q^s-1}{q-1}$. In \cite{KRS}, it is shown that $N_{q,s}$ is $K_{s,s!+1}$-free (and hence $K_{s,t}$-free).

The outline of the proof of Theorem \ref{main theorem kst} is as follows. Let $n$ be fixed and sufficiently large, and we will construct a regular $K_{s,t}$-free graph with $\Omega\left(n^{2-1/s}\right)$ edges. We use Theorem \ref{prime basis theorem} to write $n$ as a sum of $s$'th powers of primes that are almost equal. We take a disjoint union of norm-graphs whose number of vertices is equal to the $s$'th powers of the primes. We use Theorem \ref{hamiltonian theorem} to remove edges from these graphs until most vertices have the same degree and the absolute points have degree $1$ fewer. Finally, we add a matching to the absolute points to make the graph regular while making sure that it remains $K_{s,t}$-free. We now proceed with the details.

Let $n$ be fixed. By Theorem \ref{prime basis theorem}, for $n$ sufficiently large, there is a constant $c_s$ which depends only on $s$ such that we may write 
\[
n = p_1^s + p_2^s + \cdots + p_\ell^s,
\]
where each $p_j$ is a prime satisfying $|p_j - (n/\ell)^{1/s}| \leq \left((n/\ell)^{1/s} \right)^{9/10}$ and $\ell \leq c_s$. Without loss of generality, assume that $p_1 \geq p_2 \geq \cdots \geq p_\ell$. Let $G_1$ be the graph on $n$ vertices which is the disjoint union of the norm-graphs $N_{p_i, s}$ for $1\leq i \leq \ell$. For brevity, call these components $N_1,\ldots, N_\ell$. Before we can use Theorem \ref{hamiltonian theorem} to equalize the degrees between components, we must show that the norm-graphs have a good spectral gap. 

\begin{theorem}\label{norm graph spectral gap theorem}
Let $\lambda_1 \geq \lambda_2 \geq \cdots \geq \lambda_{q^s}$ the the eigenvalues of the adjacency matrix of $N_{q,s}^o$. Then for $i>1$, we have 
\[
|\lambda_i| \leq \sqrt{q^s}.
\]
\end{theorem}

\begin{proof}
Given an abelian group $\Gamma$ and a subset $S\subset \Gamma$, we define the {\em Cayley sum graph} $\textup{CayS}(\Gamma, S)$ as the graph with vertex set $\Gamma$ and $u\sim v$ if and only if $u+v\in \Gamma$. Let $S_1 := \{a\in \mathbb{F}_{q^s}:N(a) = 1\}$ be the subset of $\mathbb{F}_{q^s}$ with norm $1$. Then the norm-graph $N_{q,s}^o$ can be written as the Cayley sum graph $\textup{CayS}((\mathbb{F}_{q^s}, +), S_1)$. The eigenvalues of Cayley sum graphs are given by character sums. Given a character $\chi$ of $\Gamma$, let 
\[
\chi(S) = \sum_{x\in S}\chi(x).
\]
Then all of the eigenvalues of $\textup{CayS}(\Gamma, S)$ are given by $\chi(S)$ (when $\chi$ is real valued) or $\pm |\chi(S)|$ (if $\chi$ is complex valued) as $\chi$ ranges over all of the additive characters of $\Gamma$ (see \cite{alon}, \cite{fan}, or \cite{devos}).

Note that $S_1$ is a multiplicative subgroup of $\mathbb{F}_{q^s}^*$. The largest eigenvalue of $N_{q,s}^o$ is $|S_1|$ and corresponds to the trivial additive character. The proof is complete after bounding the remaining eigenvalues by applying the Lemma \ref{character sum lemma}.

\end{proof}

We use Theorem \ref{norm graph spectral gap theorem} to show that norm-graphs have almost regular spanning subgraphs.

\begin{theorem}\label{norm graph subgraphs}
Let $\epsilon > 0$ and $k\in \mathbb{N}$. For $q$ sufficiently large, if $k< q^{s-1-\epsilon}$ then the norm-graph $N_{q, s}$ contains spanning subgraphs with each of the following degree sequences.
\begin{enumerate}[(a)]
\item\label{subgraph a}  $q^s - \frac{q^s-1}{q-1}$ vertices that have degree 
\[
\frac{q^s-1}{q-1} - 2k,
\]
and the remaining $\frac{q^s-1}{q-1}$ vertices have degree 
\[
\frac{q^s-1}{q-1} - 2k - 1.
\]

\item\label{subgraph b} $q^s - \frac{q^s-1}{q-1} + 1$ vertices of degree
\[
\frac{q^s-1}{q-1} - 2k+1,
\]
and the remaining $\frac{q^s-1}{q-1} - 1$ vertices of degree 
\[
\frac{q^s-1}{q-1} - 2k.
\]

\end{enumerate}
\end{theorem}

\begin{proof}
 We construct a sequence of graphs $G_0, G_1, \ldots, G_k$ and show that $G_k$ has the degree sequence that we require. Let $G_0 = N_{q,s}$ and note tht $G_0$ has $q^s-\frac{q^s-1}{q-1}$ vertices of degree $\frac{q^s-1}{q-1}$ and the remaining vertices (the absolute points) have degree $1$ fewer. If each $G_j$ is Hamiltonian, then remove a Hamilton cycle from it to create $G_{j+1}$ until we reach $G_{k-1}$. For part (\ref{subgraph a}), remove a Hamilton cycle from $G_{k-1}$ and for part (\ref{subgraph b}), remove a matching on $q^s-1$ vertices where
the one vertex not incident to an edge in the 
matching is an absolute point (this must exist if $G_{k-1}$ is Hamiltonian). Since $G_k$ has the correct degree sequence, it suffices to show that each $G_j$ is Hamiltonian.

We prove that each $G_j$ is Hamiltonian using Theorem \ref{hamiltonian theorem} and induction. Let $0=\mu_0(G_j)\leq \mu_1(G_j) \leq \cdots \leq \mu_{q^s-1}(G_j)$ be the eigenvalues of the combinatorial Laplacian of $G_i$. Let $d(G_j)$ be the average degree of $G_j$ and note that because $k< q^{s-1-\epsilon}$, we have $d(G_j) \sim q^{s-1}$. Therefore, if 
\begin{equation}\label{spectral gap equation}
|d(G_j) - \mu_i(G_j)| \leq q^{s/2} + 1 + 6j = O\left( q^{s-1-\epsilon}\right),
\end{equation}
for all $i\not=0$, then we may apply Theorem \ref{hamiltonian theorem} to conclude that $G_j$ is Hamiltonian. We prove \eqref{spectral gap equation} by induction. Let $i>0$ be fixed. When $j=0$, notice that the combinatorial Laplacians of $N_{q,s}$ and $N_{q,s}^o$ are in fact the same matrix. That is
\[D(G_0) - A(G_0) = D(N_{q,s}) - A(N_{q,s}) = D(N_{q,s}^o) - A(N_{q,s}^o) = \left(\frac{q^s-1}{q-1}\right) I - A(N_{q,s}^o).
\]

By Theorem \ref{norm graph spectral gap theorem}, this implies that we have 
\[
\left| \frac{q^s-1}{q-1} - \mu_i(N_{q,s})\right| \leq q^{s/2}.
\]
Since the average degree of $N_{q,s}$ is between $\frac{q^s-1}{q-1}-1$ and $\frac{q^s-1}{q-1}$, we have that $|d(G_0) - \mu_i(G_0)| \leq q^{s/2}+1$. Now assume that \eqref{spectral gap equation} holds for $G_{j-1}$. Note that $d(G_j) = d(G_{j-1}) - 2$ and \[
(D-A)(G_j) = (D-A)(G_{j-1}) - 2I + A(C_{q^s})
\]
where $C_{q^s}$ is a cycle on $q^s$ vertices. 
By the Courant-Weyl inequalities, we have that $|\mu_i(G_j) - \mu_i(G_{j-1})|$ is bounded above by the spectral radius of $2I - A(C_{q^s})$ which is less than $4$. By the triangle inequality,
\[
|d(G_j) - \mu_i(G_j)| \leq |d(G_{j-1}) - \mu_i(G_{j-1})| + 6.
\]
Applying the induction hypothesis completes the proof.
\end{proof}

We now apply Theorem \ref{norm graph subgraphs} to each component of $G_1$, using that the primes $p_j$ all satisfy $|p_j - (n/\ell)^{1/s}| \leq \left((n/\ell)^{1/s} \right)^{9/10}$. If $n$ or $s$ is even, then apply part (\ref{subgraph a})  to find a subgraph of $G_1$ so that all of the non-absolute points have degree $\frac{p_\ell^s-1}{p_\ell-1}$ and all of the absolute points have degree $\frac{p_\ell^s-1}{p_\ell-1} -1$. If both $n$ and $s$ are odd, then apply part (\ref{subgraph b}) so that in the $j$'th component of $G_1$, we have $p_j^s - \frac{p_j^s-1}{p_j-1} + 1$ vertices of degree $\frac{p_\ell^s - 1}{p_\ell-1} -1$ and the remaining vertices have degree $\frac{p_\ell^s - 1}{p_\ell -1} -2$. 

Call this graph $G_2$. In either case, the number of vertices of minimum degree in $G_2$ is even. All that remains is to ``fix" the vertices of minimum degree. To do this we will add a matching to the minimum degree vertices of $G_2$ such that each edge has one endpoint in $N_i$ and one endpoint in $N_j$ for some $i\not=j$. This is accomplished with the following lemma (see, for example, \cite{sitton}).

\begin{lemma}\label{multipartite matching lemma}
Let $n_1\geq \cdots \geq n_\ell$ be natural numbers satisfying with $\ell\geq 3$ and $n_1 < n_2 + \cdots + n_\ell$ and $n_1 + \cdots +n_\ell$ even. Then the complete multipartite graph $K_{n_1,\cdots, n_\ell}$ contains a perfect matching.
\end{lemma}

Since we have ensured that the number of minimum degree vertices in $G_2$ is even and since the number of these vertices in each component is asymptotically equal, we may apply Lemma \ref{multipartite matching lemma} to add a matching to the minimum degree vertices of $G_2$ where each edge has endpoints in two components of $G_2$. Call this graph $G_3$ which is either $\frac{p_\ell^s - 1}{p_\ell -1}$ or $\left(\frac{p_\ell^s - 1}{p_\ell -1}-1\right)$-regular, depending on the parity of $n$ and $s$. Since $\ell$ is upper bounded by a constant depending only on $s$, we have that the degree of regularity is $\Omega(n^{1-1/s})$ where the implicit constant depends only on $s$.

The proof is complete once we show that $G_3$ is $K_{s,t}$-free. Since $G_2$ was $K_{s,t}$ free, any potential $K_{s,t}$ must contain an edge of the matching that we added. Let $uv$ be this edge and assume that $N_u$ and $N_v$ are the respective components that $u$ and $v$ are in in $G_2$. Assume that $u$ is in the part of the $K_{s,t}$ that has $s$ vertices. Then because we only added a matching to $G_2$, the remaining $t-1$ vertices in the part of size $t$ must belong to $N_u$ and the remaining $s-1$ vertices in the part of size $s$ must belong to $N_v$. Since $s-1$ and $t-1$ are at least $2$, and there is only a matching between components, this is a contradiction, and the $K_{s,t}$ cannot exist.

\section*{Acknowledgements}
The first author would like to thank Qing Xiang for teaching him about Gauss sums.  Both authors
thank Yair Caro for help in accurately 
presenting the evolution of regular Tur\'{a}n numbers.  

In an earlier version of this paper, the authors overlooked a significanly simpler argument 
and this was pointed out by Krivelevich \cite{K}.  Instead of using the fact that 2-factors exist in 
regular even degree graphs, a spectral approach similar to the proof of Theorem \ref{main theorem kst} was taken.  While
we still find the results obtained using that more complicated method (counting copies of $C_4$ in Brown's
$K_{3,3}$-free graph and then showing that it is an expander) interesting, we have removed this unnecessary work. Much thanks to Michael Krivelevich for
useful comments regarding our original proofs of Theorems \ref{main theorem k33} and 
\ref{main theorem kst}.



\begin{thebibliography}{10}
\bibitem{alon2}
N.\ Alon,
Explicit Ramsey graphs and orthonormal labelings,
{\em Electron. J. Combin.}, 1: R12, (1994), 1--8.

\bibitem {alon}
N.\ Alon,
Large sets in finite fields are sumsets,
{\em J.\ Number Theory},
{\bf 126} (2007), no.\ 1, 110--118.

\bibitem{AR}
N.\ Alon, V.\ R\"odl,
Sharp bounds for some multicolor Ramsey numbers,
{\em Combinatorica}, {\bf 25}, (2005), 125--141.

\bibitem{AFK}
N.\ Alon, S.\ Friedland, G.\ Kalai,
Regular subgraphs of almost regular graphs,
{\em J. Comb. Theory Ser. B} {\bf 37} (1), (1984), 79--91.

\bibitem{A}
B.\ Andr\'asfai,
Graphentheoretische Extremalprobleme,
{\em Acta Math. Acad. Sci. Hungar.} {\bf 15}, (1964), 413–418.

\bibitem{baker}
R.\ C.\ Baker, G.\ Harman, J.\ Pintz, 
The Difference Between Consecutive Primes, II.
{\em Proc.\ London Math.\ Soc}.\ (3) 83 (2001), no.\ 3, 532--562.

\bibitem{BIP}
A.\ Bishnoi, F.\ Ihringer, V.\ Pepe,
A Construction for Clique-Free Pseudorandom Graphs
{\em Combinatorica}, (2020).

\bibitem{brown}
W.\ G.\ Brown, 
On graphs that do not contain a Thomsen graph, 
{\em Canad.\ Math.\ Bull}.\ 9 (1966), 281--285.

\bibitem{bc}
R.\ C.\ Bose, S.\ Chowla,
Theorems in the additive 
theory of numbers, 
{\em Comment.\ Math.\ 
Helv}.\ 37 (1962/63), 141--147.

\bibitem{BC}
S.\ Butler, F.\ Chung,
Small Spectral Gap in the Combinatorial Laplacian Implies Hamiltonian,
{\em Ann.\ Comb.} {\bf 13} (2010), 403--412.

\bibitem{CT}
Y.\ Caro, Zs.\ Tuza, Regular Tur\'an numbers,
arXiv preprint, arXiv:1911.00109
(2019). 

\bibitem{singular}
Y.\ Caro, Zs.\ Tuza, 
Singular Ramsey and Tur\'{a}n numbers, 
{\em Theory Appl.\ Graphs} 6 (2019), no.\ 1, 
Art.\ 1.  

\bibitem{CJK}
S.\ Cambie, R.\ de Joannis de Verclos, R.\ Kang,
Regular Tur\'an numbers and some Gan-Loh-Sudakov-type problems, arXiv preprint arXiv:1911.08452 (2019).

\bibitem{fan}
F.\ R.\ K.\ Chung,
Diameters and eigenvalues,
{\em J. Amer. Math. Soc.},
{\bf 2}, (1989), no.\ 2, 187--196.

\bibitem{CLL}
S.\ M.\ Cioab\u{a}, F.\ Lazebnik, W.\ Li,
On the spectrum of Wenger graphs,
{\em J.\ Comb.\ Theory Ser.\ B}, 
{\bf 107}, (2014), 132--139.

\bibitem{devos}
M.\ DeVos, L.\ Goddyn, B.\ Mojar, R.\ \v{S}\'amal,
Cayley sum graphs and eigenvalues of $(3,6)$-fullerness,
{\em J.\ Comb.\ Theory Ser.\ B},
{\bf 99} (2), (2009), 358--369.


\bibitem{EJ} 
G.\ Exoo, R.\ Jajcay,
Dynamic cage survey,
{\em Electron. J. Combin.} {\bf 15}(16), (2008), P4.

\bibitem{F}
Z.\ F\"uredi,
New asymptotics for bipartite Tur\'an numbers,
{\em J. Combin. Theory Ser. A}, {\bf 15}(1), (1996), 141--144.

\bibitem{exact}
D.\ Gerbner, B.\ Patk\'{o}s, Zs.\ Tuza, M.\ Vizer, 
Some exact results for regular 
Tur\'{a}n problems,
arXiv preprint arXiv:1912.1028 (2019).

\bibitem{worm}
D.\ Gerbner, B.\ Patk\'{o}s, Zs.\ Tuza, M.\ Vizer, 
Singular Tur\'{a}n numbers and 
WORM-colorings, 
 arXiv preprint arXiv:1909.04980 (2019).

\bibitem{KRS}
J.\ Koll\'ar, L.\ R\'onyai, and T.\ Szab\'o,,
Norm-graphs and Bipartite Tur\'an Numbers,
{\em Combinatorica}, {\bf 16} (3), (1996), 399--406.

\bibitem{sasha}
A.\ V.\ Kostochka, A.\ Raspaud, B.\ Toft, D.\ B.\ West, D.\ Zirlin,
Cut-edges and regular factors in regular graphs of odd degree,
{\texttt arXiv:1806.05347} (2018).

\bibitem{K}
M.\ Krivelevich, personal communication.


\bibitem{KSsurvey}
 M.\ Krivelevich, B.\ Sudakov, 
 Pseudo-random graphs, in: 
 {\em More Sets, Graphs and Numbers, Bolyai Society Mathematical Studies} 15, Springer, (2006), 199-262.

\bibitem{kumchev}
A.\ Kumchev, H.\ Liu, 
On sums of powers of almost equal primes, 
{\em J.\ Number Theory}
176 (2017), 344--364.  

\bibitem{ADG survey}
F.\ Lazebnik, S.\ Sun, Y.\ Wang,
Some graphs, hypergraphs and digraphs defined by systems of equations: a survey,
{\em Lecture Notes of Seminario Interdisciplinare di Matematica} {\bf 14}, (2017), 105--142.

\bibitem{lidl}
R.\ Lidl, H.\ Niederreiter,
Finite Fields.  Second 
Edition, Encyclopedia of Mathematics and its Applicaitons, 20.  
{\em Cambridge University Press, Cambridge}, 1997.  

\bibitem{MSW}
G.\ E.\ Moorhouse, S.\ Sun, J.\ Williford,
The eigenvalues of the graphs $D(4,q)$,
{\em J.\ Combin.\ Theory Ser.\ B}, {\bf 125}, (2017), 1--20.

\bibitem{mors}
M.\ M\"ors, 
A new result on the problem of Zarankiewicz,
{\em J. Combin. Theory Ser. A}, {\bf 31} (2), (1981), 126--130.

\bibitem{zoli}
Z.\ L.\ Nagy,
Supersaturation of $C_4$: From Zarankiewicz towards Erd\H{o}s-Simonovits-Sidorenko,
{\em Europ. J. Combin.} {\bf 75}, (2019), 19--31.

\bibitem{parsons}
T.\ D.\ Parsons,
Graphs from projective planes,
{\em Aequationes Math}.\ 
14 (1976), no.\ 102, 167--189.

\bibitem{ruzsa}
I.\ Ruzsa,
Solving a linear equation in a set of integers.\ I.,
\emph{Acta Arith}.\ 65 (1993), no.\ 3, 259--282.

\bibitem{sitton}
D.\ Sitton,
Maximum matchings in complete multipartite graphs,
{\em Furman University Electron.\ J. Undergrad.\ Math.}
2.1, (1996), 6--16.

\bibitem{szabo}
T.\ Szab\'o, 
On the spectrum of projective norm-graphs,
{\em Information processing letters},
86(2), 2003), 71--74.

\bibitem{tt}
M.\ Tait, C.\ Timmons,
The Zarankiewicz problem in 3-partite graphs, 
{\em J.\ Combin.\ Des}.\ 
27 (2019), no.\ 6, 
391--405.


\bibitem{WW}
B.\ Wei, T.\ D.\ Wooley,
On sums of powers of almost equal primes,
{\em Proc.\ London Math.\ Soc.},
{\bf 111} (5), (2015), 1130--1162.

\bibitem{williford}
J.\ Williford,
Constructions in finite geometry 
with applications to graphs.
Thesis (Ph.D.) - University of 
Delaware, 2004.  

\end{thebibliography}
\end{document}